\documentclass[hidelinks,12pt,a4paper,reqno]{amsart}
\usepackage[utf8]{inputenc}
\usepackage{enumitem}
\usepackage{amsmath}
\usepackage{amsfonts}
\usepackage{amssymb}
\usepackage[all]{xy}
\usepackage{xfrac}
\usepackage[yyyymmdd,hhmmss]{datetime}
\usepackage{centernot}
\usepackage[english]{babel}
\usepackage{dsfont}
\usepackage{bbold}
\usepackage{soul}
\usepackage{lmodern}
\usepackage{comment,xcolor}

\usepackage[backref=page]{hyperref}
\renewcommand*{\backref}[1]{}
\renewcommand*{\backrefalt}[4]{%
    \ifcase #1 (Not cited.)%
    \or        (Cited on page~#2.)%
    \else      (Cited on pages~#2.)%
    \fi}

\newtheorem{theorem}{Theorem}[section]

\newtheorem{definition}{Definition}[section]
\newtheorem{proposition}[theorem]{Proposition}

\newcommand{\R}{\mathbb{R}}
\newcommand{\C}{\mathbb{C}}

\DeclareMathOperator{\End}{End}

\newcommand{\muu}{\text{\sfrac{1}{2}}}

\author{J. P. Fatelo and N. Martins-Ferreira}

\address{School of Technology and Management, Centre for Rapid and Sustainable Product Development - CDRSP, Polytechnic Institute of Leiria, P-2411-901 Leiria, Portugal.}

\email{jorge.fatelo@ipleiria.pt}

\email{martins.ferreira@ipleiria.pt}

\title[Affine mobi spaces]{Affine mobi spaces}

\date{\today; \currenttime}


\thanks{  This work is supported by Fundação para a Ciência e a Tecnologia (FCT\-UID-Multi-04044-2019), Centro2020 (PAMI -- ROTEIRO\-/0328\-/2013\-- 022158) and Polytechnic of Leiria through the projects CENTRO\--01\--0247-FEDER: 069665, 069603, 039958, 039969, 039863, 024533 and also ESTG and CDRSP}

\begin{document}

\begin{abstract}
The category of mobi algebras has been introduced as a model to the unit interval of real numbers. 
The notion of  mobi space over a mobi algebra has been proposed as a model for spaces with geodesic paths. In this paper we analyse the particular case of affine mobi spaces and show that there is an isomorphism of categories between R-modules and pointed affine mobi spaces over a mobi algebra R as soon as R is a unitary ring in which $2$ is an invertible element. 
\end{abstract}

\maketitle


\section{Introduction}

This paper is the continuation of a program which aims at modelling spaces with geodesics as purely algebraic structures by considering here the particular case of affine spaces. A geodesic from a point~$x$ to a point~$y$ in a space $X$ is in particular a path, that is, a curve $\gamma\colon{[0,1]\to X}$ with $\gamma(0)=x$ and $\gamma(1)=y$. The necessity of a suitable algebraic structure for the unit interval has lead us to the notion of mobi algebra \cite{mobi}. The unit interval can be seen as a set $A$, together with three constants $0$, $\muu$, $1$ and a ternary operation $$p(a,b,c)=a-ba+bc$$ satisfying the following eight conditions:
\begin{enumerate}[label={\bf (A\arabic*)}]
\item\label{alg_mu} $p(1,\muu,0)=\muu$
\item\label{alg_01} $p(0,a,1)=a$
\item\label{alg_idem} $p(a,b,a)=a$
\item\label{alg_0} $p(a,0,b)=a$
\item\label{alg_1} $p(b,1,a)=a$
\item\label{alg_cancel} $p(a,\muu,b)=p(a,\muu,c)\implies b=c$
\item\label{alg_homo} $p(a,p(b_1,b_2,b_3),c)=p(p(a,b_1,c),b_2,p(a,b_3,c))$
\item\label{alg_medial}
$p(p(a_1,b,c_1),\muu,p(a_2,b,c_2))=p(p(a_1,\muu,a_2),b,p(c_1,\muu,c_2))$.
\end{enumerate}
This is precisely the definition of a mobility algebra (mobi algebra for short). That is, a mobi algebra is a system $(A,p,0,\muu,1)$ consisting of a set $A$, together with a ternary operation $p\colon{A^3\to A}$ and three constants $0$, $\muu$ and $1$ that are required to satisfy the eight axioms above. We may form the category of mobi algebras by considering as morphisms the expected structure preserving maps. This means that if $(A,p,0,\muu,1)$ and $(A',p',0,\muu,1)$ are two mobi algebras then a morphism is a map $f\colon{A\to A'}$ such that 
$f(0)=0,\ f(\muu)=\muu,\ f(1)=1$, and $ f(p(a,b,c))=p'(f(a),f(b),f(c))$, for all $a,b,c\in A$. For example, the map $$f:[0,1]\to[-1,5];\quad t\mapsto 6t-1 $$ is a morphism from $([0,1],p,0,\frac{1}{2},1)$ to $([-1,5],p',-1,2,5)$ with $$p(a,b,c)=a+b(c-a)\quad \textrm{and}\quad
p'(a,b,c)=\frac{c+5a}{6}+b\,\frac{c-a}{6}.$$ Other examples are the inclusion of the unit interval into the real line or the projection of the mobi algebra considered in Example \ref{ex_lozenge} of Section~\ref{sec_examples} into the unit interval.

If there exists an element $2\in A$ such that \mbox{$p(0,\muu,2)=1$} then the structure $(A,+,\cdot,0,1)$ is a unitary ring with \mbox{$a+b=p(0,2,p(a,\muu,b))$} and $a\cdot b=p(0,a,b)$ for every $a,b\in A$ (a proof can be found in \cite{mobi}). This gives an isomorphism of categories between mobi algebras with $2$ and unitary rings with $\muu$, the inverse of $2$.

A mobi space \cite{preprint2} is defined over a mobi algebra $(A,p,0,\muu,1)$ as a pair $(X,q)$ with $X$ a set and $$q\colon{X\times A\times X\to X}$$ a map satisfying the following conditions:
\begin{enumerate}[label={\bf (X\arabic*)}]
\item\label{space_0} $q(x,0,y)=x$
\item\label{space_1} $q(y,1,x)=x$
\item\label{space_idem} $q(x,a,x)=x$
\item\label{space_cancel} $q(x,\muu,y)=q(x,\muu,y')\implies y=y'$
\item\label{space_homo} $q(q(x,a,y),b,q(x,c,y))=q(x,p(a,b,c),y)$.
\end{enumerate}
Note that axioms \ref{space_0}---\ref{space_homo} are analogous to axioms \ref{alg_idem}---\ref{alg_homo}. A natural generalization of \ref{alg_medial} would be
\begin{eqnarray}\label{affine_1}
q(q(x_1,a,y_1),\muu,q(x_2,a,y_2))=q(q(x_1,\muu,x_2),a,q(y_1,\muu,y_2)).
\end{eqnarray}
However, this condition is too restrictive and it is not in general verified when $q(x,t,y)$, with $t\in[0,1]$, is a geodesic path in an arbitrary space (see for example \cite{preprint2} for the case of the 2-sphere).
That is the reason why we do not include it in the definition of a mobi space.
The interpretation of axioms \ref{space_0} to \ref{space_cancel} is clear when $q(x,t,y)$ is the point in the line segment connecting $x$ and $y$ at $t\in [0,1]$, while the interpretation of \ref{space_homo} is depicted in the following diagram (for simplicity we have considered $X=[0,1]$ and $p=q$).
\[
\xymatrix@R-=0.5cm{
&&1\ar@{-}[d]\\
&1\ar@{--}[r]\ar@{-}[d]&y\ar@{-}[d]\\
1\ar@{-}[d]\ar@{--}[r]&c\ar@{--}[r]\ar@{-}[d]&q(x,c,y)\ar@{-}[d]\\
b\ar@{-}[d]\ar@{--}[r]&p(a,b,c)\ar@{--}[r]\ar@{-}[d]&(X5)\ar@{-}[d]\\
0\ar@{--}[r]&a\ar@{--}[r]\ar@{-}[d]&q(x,a,y)\ar@{-}[d]\\
&0\ar@{--}[r]&x\ar@{-}[d]\\
&&0}
\]

 When condition (\ref{affine_1}) is satisfied for all $x_1,x_2,y_1,y_2\in~X$ and $a\in A$, we say that the $A$-mobi space $(X,q)$ is \emph{affine} and speak of an \emph{affine $A$-mobi space}. The purpose of this paper is to show that, for a unitary ring $R$ with $\muu$ (which is the same as a mobi algebra with~$2$), the category of pointed affine $R$-mobi spaces is isomorphic to the familiar category of $R$-modules.
A pointed affine mobi space $X$ is an affine mobi space together with an element of $X$ considered as its \emph{origin}.
Every affine space in the usual sense (see e.g. \cite{Porteus} p.74) can be seen as an affine mobi space with $q(x,t,y)=x+t(y-x)$. Note that $t$ is a scalar, $(y-x)$ is a vector and $x$ is a point (see also e.g. \cite{Bertram,bourn_geometrie}). Moreover, if $R$ is a field of characteristic different from 2 then every affine mobi space is an affine space in the usual sense since, for every choice of an origin $e\in X$, the structure $(X,+,e,\varphi)$ is a vector space over $R$ with $x+y=q(e,2,q(x,\muu,y))$ and $\varphi_t(x)=q(e,t,x)$ (Theorem \ref{mobi2module}).

In the same way as for modules over a ring, if fixing a mobi algebra $(A,p,0,\muu,1)$ then we may consider the category of affine $A$-mobi spaces whose morphisms $g\colon{(X,q)\to (X',q')}$ are the maps \mbox{$g\colon{X\to X'}$} such that
\[g(q(x,a,y))=q'(g(x),a,g(y))\]
for every $x,y\in X$ and $a\in A$. Moreover, if allowing a variable mobi algebra then we may consider the category of all affine mobi spaces. In that situation the morphisms are pairs $(f,g)$ with $f\colon{A\to A'}$ a morphism of mobi algebras and $g\colon{X\to X'}$ a map such that
\[g(q(x,a,y))=q'(g(x),f(a),g(y))\]
for every $x,y\in X$ and $a\in A$. Note that $X$ is an affine mobi space over $A$ while $X'$ is an affine mobi space over $A'$. 
When $A$ is a field (of characteristic different from 2) and $X$ is a vector space, the previous condition with $q(x,a,y)=(1-a)x+ay$ is
\[g((1-a)x+ay)=(1-f(a))g(x)+f(a)g(y),\] %
while the homologous condition in vector spaces is
\[g(a(x+y))=f(a)(g(x)+g(y))\]
for all vectors $x$ and $y$ and scalars $a\in A$.

This paper is organized as follows. In Section \ref{sec_mobi_alg} the notion of  mobi algebra is recalled, together with its derived operations (further details can be found in \cite{mobi}).  Section \ref{sec_mobi_space} is devoted to the definition of mobi space and a short list of its properties. In Section \ref{sec_examples}, examples of affine and non-affine mobi spaces are presented. The main results are in Section \ref{sec:modules} where the comparison between affine mobi spaces and modules over a ring is detailed. In Section \ref{isomorphism} we formalize the corresponding isomorphism between the two categories.

\section{Mobi algebra}\label{sec_mobi_alg}

In this section we briefly recall the notion of mobi algebra, introduced in \cite{mobi}, and some of its basic properties. Several examples  of different nature were presented in  \cite{mobi}. The motivating example of a mobi algebra is the unit interval with the three constants $0, \frac{1}{2}, 1\in [0,1]$ and $p(a,b,c)=a-ba+bc$. As stated in the introduction, a mobi algebra plays the role of scalars (for a mobi space) in the same way as a ring (or a field) models the scalars for a module (or a vector space) over the base ring (or field). 
In order to have an intuitive interpretation of its axioms we may consider a mobi algebra as a mobi space over itself and use the geometric intuition provided in section \ref{sec_mobi_space}.
Namely, that the operation $p(x,t,y)$ is the position of a particle moving from a point $x$ to a point $y$ at an instant $t$ while following a geodesic path.

\begin{definition}\cite{mobi}\label{mobi_algebra}
A mobi algebra is a system $(A,p,0,\muu,1)$, in which $A$ is a set, $p$ is a ternary operation and $0$, $\muu$ and $1$ are elements of A, that satisfies the axioms \ref{alg_mu}---\ref{alg_medial}.
\end{definition}

Some properties of mobi algebras can be suitably expressed in terms of a unary operation~"$\bar{()}$", which is an involution, and binary operations~"$\cdot$", "$\circ$" and "$\oplus$" defined as follows (see \cite{mobi} for more details).
\begin{definition}\label{binary_definition} Let $(A,p,0,\muu,1)$ be a mobi algebra. We define:
\begin{eqnarray}
\label{def_complementar}\overline{a}&=&p(1,a,0)\\
\label{def_product}a\cdot b&=&p(0,a,b)\\
\label{def_oplus}a\oplus b&=&p(a,\muu,b)\\
\label{def_star}a\circ b&=&p(a,b,1).
\end{eqnarray}
\end{definition}

We recall a list of some properties of a mobi algebra. If $(A,p,0,\muu,1)$ is a mobi algebra, then: 
\begin{eqnarray}
\label{Bmuu} \overline{\muu}&=&\muu\\
\label{B130} a\cdot\muu=\muu\cdot a&=&0\oplus a\\
\label{cancell-muu} a\cdot\muu=b\cdot\muu&\Rightarrow& a=b\\
\label{B7} p(\overline{a},\muu,a)&=&\muu\\
\label{Boverlinea} \overline{a}=a &\Rightarrow & a=\muu\\
\label{complementary-p} \overline{p(a,b,c)}&=&p(\overline{a},b,\overline{c})\\
\label{commut} p(c,b,a)&=&p(a,\overline{b},c)\\
\label{circ_cdot} \overline{a\circ b}&=&\overline{b}\cdot\overline{a}\\
\label{important}\muu\cdot p(a,b,c)&=&(\overline{b}\cdot a)\oplus (b\cdot c).
\end{eqnarray}
The element $\muu$ of a mobi algebra is the only element $a$ such that $a=\bar{a}$ and as a consequence it is automatically preserved by homomorphisms as soon as the rest of the structure is preserved. Alternatively, $1$ and $\bar1=0$ are preserved by morphisms as soon as $\muu$ is preserved, as explained in the following proposition.

\begin{proposition}
Consider two mobi algebras $A=(A,p,0,\muu,1)$ and $A'=(A',p',0,\muu,1)$ and a map $f:A\to A'$such that
\begin{equation}\label{homofp}
f(p(a,b,c))=p'(f(a),f(b),f(c)).
\end{equation}
Then, the following properties of $f$ are equivalent:
\begin{enumerate}
\item[(i)] $f(0)=0$ and $f(1)=1$;
\item[(ii)] $f\left(\muu\right)=\muu$.
\end{enumerate}
\end{proposition}
\begin{proof}
If (\ref{homofp}) and $(i)$ are satisfied then $p'(1,f(\muu),0)= f(\muu)$, by Axiom \ref{alg_mu}, i.e $\overline{f(\muu)}=f(\muu)$ which implies, using Property (\ref{Boverlinea}), that $f(\muu)=\muu$. Conversely, if (\ref{homofp}) and $(ii)$ are satisfied then, because $1\cdot\muu=\muu$ by \ref{alg_1}, we have $f(1)\cdot \muu=\muu=1\cdot\muu$ and therefore $f(1)=1$ due to Property (\ref{cancell-muu}). Now \ref{alg_mu} imply 
$p'(1,\muu,f(0))= \muu=p'(1,\muu,0)$ and therefore $f(0)=0$ due to \ref{alg_cancel}.
\end{proof}

The following proposition can be found (implicitly) in \cite{mobi}. We say that a unitary ring has an element called $\muu$ when the equation $x+x=1$ can be solved with $x=\muu$. Similarly, we say that a mobi algebra has an element called $2$ when the equation $p(0,\muu,x)=1$ can be solved with $x=2$.

\begin{proposition}\label{thm: mobi=ring}
There is an isomorphism between the full subcategory of mobi algebras with $2$ and the full subcategory of unitary rings with $\muu$.
\end{proposition}
\begin{proof}
If $(A,+,\cdot,0,1)$ is a unitary ring with $\muu$ then $(A,p,0,\muu,1)$ is a mobi algebra with $p(a,b,c)=a+bc-ba$. Conversely, if $(A,p,0,\muu,1)$ is a mobi algebra with $2$ then  $(A,+,\cdot,0,1)$ is a unitary ring with \mbox{$a+b=p(0,2,p(a,\muu,b))$} and $a\cdot b=p(0,a,b)$. A map $f\colon{A\to A'}$  is a ring homomorphism if and only if it is a mobi algebra morphism provided the ring has a $\muu$ and the mobi algebra has a $2$. The details can be found in \cite{mobi}.
\end{proof}

A possible generalization  of Axiom \ref{alg_medial} is 
\begin{eqnarray}\label{affine_2}
p(p(x_1,a,y_1),b,p(x_2,a,y_2))=p(p(x_1,b,x_2),a,p(y_1,b,y_2))
\end{eqnarray}
which must hold for all $b$ and not only for $b=\muu$. The following proposition is a particular instance of Proposition 6.4 in \cite{mobi}.

\begin{proposition}
The full subcategory of mobi algebras with $2$ and satisfying condition (\ref{affine_2}) instead of \ref{alg_medial}, is isomorphic to the category of commutative unitary rings with $\muu$.
\end{proposition}
\begin{proof}
In a mobi algebra with $2$, the condition (\ref{affine_2}) is equivalent to the condition $p(0,a,b)=p(0,b,a)$ for all $a,b\in A$.
\end{proof}

Since we are not necessarily interested in commutative rings, we choose to consider condition  (\ref{affine_1}) as a direct translation of axiom \ref{alg_medial} rather than as a translation of condition (\ref{affine_2}). Note that if condition~(\ref{affine_1}) is generalized to 
\begin{eqnarray*}
q(q(x_1,a,y_1),b,q(x_2,a,y_2))=q(q(x_1,b,x_2),a,q(y_1,b,y_2)),
\end{eqnarray*}
then it offers a correspondence with the axioms for an affine space in~\cite{Bertram}. 

\section{Mobi space}\label{sec_mobi_space}

In this section we recall the definition of a mobi space over a mobi algebra. Its main purpose is to serve as a model for spaces with a geodesic path connecting any two points. 
We call \emph{affine} to a A-mobi space $X$ verifying  condition~(\ref{affine_1}).

\begin{definition}\label{mobi_space}
Let $(A,p,0,\muu,1)$ be a mobi algebra. An $A$-mobi space $(X,q)$ consists of a set $X$ and a map $q\colon{X\times A\times X\to X}$ such that the conditions  \ref{space_0}---\ref{space_homo} are satisfied.

\end{definition}


\begin{definition}\label{affine_mobi_space}
Let $(A,p,0,\muu,1)$ be a mobi algebra. An affine $A$-mobi space $(X,q)$ is a $A$-mobi space such that condition (\ref{affine_1}) is satisfied for all $x_1,x_2,y_1,y_2\in X$ and $a\in A$.
\end{definition}

Here are some immediate consequences of the axioms for a mobi space, not necessarily affine.

\begin{proposition}\label{properties_space} Let $(A,p,0,\muu,1)$ be a mobi algebra and $(X,q)$ an A-mobi space. It follows that:
\begin{enumerate}[label={\bf (Y\arabic*)}]
\item\label{Y1} $q(y,a,x)=q(x,\overline{a},y)$
\item\label{Y2} $q(y,\muu,x)=q(x,\muu,y)$
\item\label{Y3} $q(x,a,q(x,b,y))=q(x,a\cdot b,y)$
\item\label{Y4} $q(q(x,a,y),b,y)=q(x,a\circ b,y)$
\item\label{Y7} $q(q(x,a,y),\muu,q(x,b,y))=q(x,a\oplus b,y)$
\item\label{Y6} $q(x,\muu,q(x,a,y))=q(x,a,q(x,\muu,y))$
\item\label{Y5} $q(q(x,a,y),\muu,q(y,a,x))=q(x,\muu,y)$
\item\label{Y8} $q(q(q(x,a,y),b,x),\muu,q(x,b,q(x,c,y)))\\
=q(x,\muu,q(x,p(a,b,c),y))$
\item\label{Y9} $q(x,a,y)=q(y,a,x)\Rightarrow q(x,a,y)=q(x,\muu,y)$
\item\label{Y10} $q(x,a,y)=q(x,b,y)\Rightarrow q(x,p(a,t,b),y)=q(x,a,y)$, \text{for all $t$}.
\end{enumerate}
\end{proposition}

\begin{proof}
The following proof of \ref{Y1}, bearing in mind (\ref{def_complementar}), uses \ref{space_homo}, \ref{space_1} and \ref{space_0}:
\begin{eqnarray*}
q(x,\overline{a},y)&=&q(x,p(1,a,0),y)\\
                   &=&q(q(x,1,y),a,q(x,0,y))\\
                   &=&q(y,a,x).
\end{eqnarray*}
\ref{Y2} follows directly from \ref{alg_mu} and \ref{Y1}.
Beginning with (\ref{def_product}), \ref{Y3} is a consequence of \ref{space_homo} and \ref{space_0}:
\begin{eqnarray*}
q(x,a\cdot b,y)&=&q(x,p(0,a,b),y)\\
          &=&q(q(x,0,y),a,q(x,b,y))\\
          &=&q(x,a,q(x,b,y)).
\end{eqnarray*}
Considering (\ref{circ_cdot}), property \ref{Y4} follows from \ref{Y1} and \ref{Y3}.
\begin{eqnarray*}
q(q(x,a,y),b,y)&=&q(y,\overline{b},q(y,\overline{a},x))\\
               &=&q(y,\overline{b}\cdot\overline{a},x)\\
               &=&q(y,\overline{a\circ b},x)\\
               &=&q(x,a\circ b,y).
\end{eqnarray*}
Considering $(\ref{def_oplus})$, \ref{Y7} is just a particular case of \ref{space_homo}.
To prove \ref{Y6}, we use \ref{space_homo}, (\ref{B130}) and \ref{space_0}.
\begin{eqnarray*}
q(x,\muu,q(x,a,y))&=&q(q(x,0,y),\muu,q(x,a,y))\\
                 &=&q(x,p(0,\muu,a),y)\\
                 &=&q(x,p(0,a,\muu),y)\\
                 &=&q(q(x,0,y),a,q(x,\muu,y))\\
                 &=&q(x,a,q(x,\muu,y)).
\end{eqnarray*}
The following proof of \ref{Y5} is based on \ref{Y1}, \ref{space_homo} and (\ref{B7});
\begin{eqnarray*}
q(q(x,a,y),\muu,q(y,a,x))&=&q(q(x,a,y),\muu,q(x,\overline{a},y))\\
                        &=&q(x,p(a,\muu,\overline{a}),y)\\
                        &=&q(x,\muu,y).
\end{eqnarray*}
To prove \ref{Y8}, we start with the important property (\ref{important}) of the underlying mobi algebra and then get:
\begin{eqnarray*}
& &q(x,\muu \cdot p(a,b,c),y)=q(x,\overline{b}\cdot a\oplus b\cdot c,y)\\
&\Rightarrow& q(x,p(0,\muu,p(a,b,c)),y)=q(x,p(\overline{b}\cdot a, \muu,b\cdot c),y)\\
&\Rightarrow& q(x,\muu,q(x,p(a,b,c),y))=q(q(x,\overline{b}\cdot a,y),\muu,q(x,b\cdot c,y))\\
&\Rightarrow& q(x,\muu,q(x,p(a,b,c),y))\\
& &=q(q(x,\overline{b},q(x,a,y)),\muu,q(x,b,q(x,c,y))).
\end{eqnarray*}
It is easy to see that $\ref{Y9}$ is a direct consequence of $\ref{Y5}$ and \ref{space_idem}, while $\ref{Y10}$ is a consequence of \ref{space_idem} and \ref{space_homo}. 

\end{proof}


\section{Examples}\label{sec_examples} 

In this section, several examples of mobi spaces are presented. In each case, besides the mobi algebra $(A, p, 0, \muu, 1)$, we present a set~$X$ and a ternary operation $q(x,a,y)\in X$, for all $x,y\in X$ and \mbox{$a\in A$}, verifying the axioms of Definition \ref{mobi_space}.
When the underlying mobi algebra structure corresponds to $A=[0,1]$ with the three constants $0, \frac{1}{2}, 1$ and $p(a,b,c)=a-b a+b c$ for all $a,b,c\in A$, we refer to it as the {\it canonical mobi algebra}. When this canonical algebra is extended to $A=\R$ we refer to the mobi algebra as the {\it real line}. The first five examples are borrowed from \cite{preprint2} where each example is inserted in an appropriate context. In particular, Example \ref{projectiles} is related to projectile motion in $\R^n$ and Example \ref{criticaldamping} with the motion of a critically damped harmonic oscillator.
The list of examples follows:

\begin{enumerate}[label={\arabic*}, leftmargin=*]
\item\label{example1} - Vector spaces provide examples of mobi spaces over the canonical mobi algebra. For instance, the \emph{canonical mobi space} on $\R^n$ corresponds to
\[X=\R^n\quad (n\in\mathbb{N})\]
and
\[q(x,a,y)=x+a\,(y-x).\]

\item\label{ex11}\label{projectiles} - For any $k\in\R^n$, we have a mobi space $(X,q)$ over the canonical mobi algebra by taking the set
\[X=\mathbb{R}^{n+1}\]
with the formula
\begin{equation*}\label{eqprojectiles}
\begin{array}{l}
q((x,s),a,(y,t))=\\[0.3mm]
\left(x+a (y-x)+a (1-a)(t-s)^2 k, s+a (t-s)\right).
\end{array}
\end{equation*} 


\item\label{criticaldamping} - For any $\alpha\in\R$, we obtain the following mobi space $(\mathbb{R}^{n+1},q)$ over the canonical mobi algebra 
with the formula
$$
\begin{array}{l}
q((x,s),a,(y,t))=\\[3mm]
\left((1-a)\,x\,e^{\alpha a (t-s)}+a\, y\,e^{\alpha (1-a) (s-t)}, s+a(t-s)\right).
\end{array}$$

\item\label{ex-general} -
The previous three examples can be extended to the real line which will be convenient when used to illustrate the theorems of Section~\ref{sec:modules} where the inverse of $\muu$ is needed. The previous examples are in fact all isomorphic. Indeed, they can be obtained from the canonical example on $X=\R^{n+1}$ with
$q((x,s),a,(y,t))=(1-a) (x,s)+a (y,t)$
through a bijective map $F:X\to X$ leading up to the new structure:
$$q((x,s),a,(y,t))=F\left((1-a)F^{-1}(x,s)+a F^{-1}(y,t)\right).$$
If we consider that $F(x,s)=\left(\lambda(s) x-K(s),s\right)$, with $\lambda(s)\neq0$ \mbox{$\forall s\in \R$}, 
we get
\begin{eqnarray*}
q((x,s),a,(y,t))&=&\Big(\lambda\left((1-a)s+a t\right)\,(1-a)\ \dfrac{x+K(s)}{\lambda(s)}\\
&&+\lambda\left((1-a)s+a t\right)\, a\ \dfrac{y+K(t)}{\lambda(t)}\\
&&-K\left((1-a)s+a t\right),(1-a)s+a t\Big).
\end{eqnarray*}
Example \ref{projectiles} is reached with $\lambda(t)=1$ and $K(t)=k t^2$ while Example~\ref{criticaldamping} is obtained for $\lambda(t)=e^{\alpha t}$ and $K(t)=0$.

\medskip
\item\label{ex7} - Considering the set $X=\mathbb{R^n}\times\mathbb{R}^+$ with the formula
\[q((x,s),a,(y,t))=\left(x+(y-x)\frac{a\,t}{s+a(t-s)},s+a(t-s)\right),\]
$(X,q)$ is a mobi space over the canonical mobi algebra.
This example is isomorphic to the canonical mobi space restricted to $X=\R^n\times\R^+$ through the map 
$F(x,s)=\left(\dfrac{x}{s},s\right)$, using the notation of Example~\ref{ex-general}. In this case, however, the extension to $a\in\R$ is not possible as $s+a(t-s)$ could then lay outside $\R^+$.

\medskip
\item\label{ex12}\label{ex_lozenge}
- So far, the examples are mobi spaces over one-dimensional mobi algebras. 
Consider now the mobi algebra $(A,p,0,\muu,1)$ with, for any $k\in\R^+$,
$$A=\left\{(a_1,a_2)\in\mathbb{R}^2\colon \sqrt{k}\,\vert a_2\vert \leq a_1 \leq 1-\sqrt{k}\,\vert a_2\vert\right\}$$
$$0=(0,0)\,;\,\muu=\left(\frac{1}{2},0\right)\,;\,1=(1,0)$$
\begin{eqnarray*}
p(a,b,c)&=&\big(a_1+b_1 (c_1-a_1)+k\, b_2 (c_2-a_2),\\
         &&\ a_2+b_1 (c_2-a_2)+b_2 (c_1-a_1)\big).
\end{eqnarray*}
Then $(X,q)$, with
$X=[0,1]$ and $$q(x,(a_1,a_2),y)= x+(a_1\pm\sqrt k\,a_2)(y-x)$$ is a mobi space.

\medskip
\item\label{ex_lozenge2}
- The operation $p$ of the previous example can be considered for any $k\in \R$. Then, for $A=\R^2$, $0=(0,0)$, $\muu=(\frac{1}{2},0)$ and $1=(1,0)$, 
$(A,p,0,\muu,1)$ is a mobi algebra. Note that, in this case, the product defined in~(\ref{def_product}) is given by
$$(a_1,a_2)\cdot(b_1,b_2)=p(0,a,b)=(a_1 b_1+k\,a_2 b_2,a_1 b_2+a_2 b_1),$$
which corresponds, for $k=-1$, $k=0$ and $k=1$, to the multiplication of complex numbers, dual numbers and split-complex numbers respectively.
This suggests to consider $$X=\left\{x_1+h\,x_2\mid \, x_1,x_2\in\R,\, h^2=k\right\}$$ 
and $$q(x,(a_1,a_2),y)= x+(a_1+h\,a_2)(y-x),$$
getting, indeed, a mobi space $(X,q)$ over $(\R^2,p,0,\muu,1)$.

\medskip
\item\label{ex-noncommut} - Consider the mobi algebra $\big(\R^3,p,(0,0,0),\left(\frac{1}{2},0,\frac{1}{2}\right),(1,0,1)\big)$ with
\[\begin{array}{ll}
p(a,b,c)=\big(&a_1+b_1 (c_1-a_1),a_2+b_1 (c_2-a_2)+b_2 (c_3-a_3),\\
         &a_3+b_3 (c_3-a_3)\big).
\end{array}
\]
Then $\left(\R^2,q\right)$ is a mobi space over this mobi algebra with
$$q(x,a,y)=\big(x_1+a_1(y_1-x_1)+a_2(y_2-x_2),x_2+a_3(y_2-x_2)\big).$$
Remark: The multiplication $p(0,a,b)=\left(a_1 b_1,a_1 b_2+a_2 b_3,a_3 b_3\right)$ of the mobi algebra is not commutative and therefore this affine mobi space is not isomorphic to the canonical one.

\medskip
\item\label{ex-nonaffine} - Consider the set
$X=\mathbb{C}^n\times\R$
and the following ternary operation on $X$ ($i$ is the imaginary unit):
\begin{equation*}
q((x,s),a,(y,t))=
\left(x+(y-x)\,a\,\dfrac{(2-a) s+a t+i}{s+t+i},s+a(t-s)\right).
\end{equation*}
Then $(X,q)$ is a mobi space over the real line.
This operation $q$ does not verify the affine condition (\ref{affine_1}) because, for instance,
$$\begin{array}{ll}
&q\big(q\left((0,0),-\frac{1}{2},(0,1)\right),\frac{1}{2},q\left((x,0),-\frac{1}{2},(0,0)\right)\big)\\[4pt]
&-q\big(q\left((0,0),\frac{1}{2},(x,0)\right),-\frac{1}{2},q\left((0,1),\frac{1}{2},(0,0)\right)\big)\\[4pt]
&=\left(\frac{3}{20} x,0\right)
\end{array}$$

\item\label{ex-nonaffine2} - We can consider the set
\[X=\mathbb{R}^{n+1}\]
and the formula
$$
\begin{array}{l}
q((x,s),a,(y,t))=\\[3mm]
\left(x+a\,(y-x)\,\dfrac{3\,s^2+3\,s(t-s)a+(t-s)^2\,a^2+1}{s^2+s\,t+t^2+1},s+a(t-s)\right).
\end{array}
$$
$(X,q)$ is a non-affine mobi space over the canonical mobi algebra or over the real line. 
\end{enumerate}

All the examples of this section, except Examples~\ref{ex-nonaffine} and~\ref{ex-nonaffine2}, verify the affine condition (\ref{affine_1}).
In \cite{preprint2}, the formula for geodesics on the $n$-sphere in terms of a mobi space structure as well as other examples of non-affine mobi spaces are given.
In the next two sections we analyse the case of affine mobi spaces and compare them with modules over a ring with one-half.

\section{Comparison  with R-modules}\label{sec:modules}

Consider a unitary ring $(R,+,\cdot,0,1)$. It has been proven \cite{mobi} that if $R$ contains the inverse of $1+1$, then it is a mobi algebra and if a mobi algebra $(A,p,0,\muu,1)$ contains the inverse of $\muu$, in the sense of the operation defined in (\ref{def_product}), then it is a unitary ring. In this section, we will compare a module over a ring with a mobi space over a mobi algebra. First, let us just recall that a module over a ring $R$ is a system $(M,+,e,\varphi)$, where $\varphi:R\to \End(M)$ is a map from $R$ to the usual ring of endomorphisms, such that $(M,+,e)$ is an abelian group and $\varphi$ is a ring homomorphism.

The following theorem shows how to construct a mobi space from a module over a ring containing the inverse of $2$.
\begin{theorem}\label{module2mobi}
Consider a module $(X,+,e,\varphi)$ over a unitary ring $(A,+,\cdot,0,1)$. If $A$ contains $(1+1)^{-1}=\muu$ then $(X,q)$ is an affine mobi space over the mobi algebra $(A,p,0,\muu,1)$, with
\begin{eqnarray}
p(a,b,c)&=&a+b c-b a\\
q(x,a,y)&=&\varphi_{1-a}(x)+\varphi_a(y).\label{q}
\end{eqnarray}
\end{theorem}
\begin{proof}
$(A,p,0,\muu,1)$ is a mobi algebra by Theorem 7.2 of \cite{mobi}. We show here that the axioms of Definition \ref{mobi_space}, as well as (\ref{affine_1}), are verified. The first three axioms are easily proved:
\begin{eqnarray*}
q(x,0,y)&=&\varphi_1(x)+\varphi_0(y)=x+ e=x\\
q(x,1,y)&=&\varphi_0(x)+\varphi_1(y)=e+ y=y\\
q(x,a,x)&=&\varphi_{1-a}(x)+\varphi_a(x)=\varphi_{1-a+a}(x)=\varphi_1(x)=x.
\end{eqnarray*}
Axiom \ref{space_cancel} is due to the fact that $\muu+\muu=1$ and consequently
\begin{eqnarray*}
\varphi_\muu(y_1)=\varphi_\muu(y_2)&\Rightarrow& \varphi_\muu(y_1)+\varphi_\muu(y_1)=\varphi_\muu(y_2)+\varphi_\muu(y_2)\\                                   &\Rightarrow& \varphi_1(y_1)=\varphi_1(y_2)\Rightarrow y_1=y_2.
\end{eqnarray*}
Next, we give a proof of Axiom \ref{space_homo}. It is relevant to notice that, besides other evident properties of the module $X$, the associativity of~$+$ plays an important part in the proof:
\begin{eqnarray*}
&&q(q(x,a,y),b,q(x,c,y))\\
&=&\varphi_{1-b}(\varphi_{1-a}(x)+\varphi_a(y))+\varphi_b(\varphi_{1-c}(x)+\varphi_c(y))\\
&=&\varphi_{1-b}(\varphi_{1-a}(x))+\varphi_{1-b}(\varphi_a(y))+\varphi_b(\varphi_{1-c}(x))+\varphi_b(\varphi_c(y))\\
&=&\varphi_{(1-b)(1-a)}(x)+\varphi_{b(1-c)}(x)+\varphi_{(1-b)a}(y)+\varphi_{bc}(y)\\
&=&\varphi_{1-a+ba-bc}(x)+\varphi_{a-ba+bc}(y)\\
&=&\varphi_{1-p(a,b,c)}(x)+\varphi_{p(a,b,c)}(y)\\
&=&q(x,p(a,b,c),y).
\end{eqnarray*}
It remains to prove (\ref{affine_1}):
\begin{eqnarray*}
&&q(q(x_1,a,y_1),\muu,q(x_2,a,y_2))\\
&=&\varphi_{\muu}(\varphi_{1-a}(x_1)+\varphi_a(y_1))+\varphi_\muu(\varphi_{1-a}(x_2)+\varphi_a(y_2))\\
&=&\varphi_{\muu}(\varphi_{1-a}(x_1)+\varphi_a(y_1)+\varphi_{(1-a)}(x_2)+\varphi_a(y_2))\\
&=&\varphi_{\muu}(\varphi_{1-a}(x_1+ x_2)+\varphi_a(y_1+ y_2))\\
&=&\varphi_{(1-a)\muu}(x_1+ x_2)+\varphi_{a\muu}(y_1+ y_2)\\
&=&\varphi_{(1-a)}(\varphi_\muu(x_1)+\varphi_\muu(x_2))+\varphi_a(\varphi_{\muu}(y_1)+\varphi_{\muu}(y_2))\\
&=&\varphi_{(1-a)}(q(x_1,\muu,x_2))+\varphi_{a}(q(y_1,\muu,y_2))\\
&=&q(q(x_1,\muu,x_2)),a,q(y_1,\muu,y_2)).
\end{eqnarray*}
\end{proof}

\begin{theorem}\label{mobi2module}
Consider an affine mobi space $(X,q)$, with a fixed chosen element $e\in X$, over a mobi algebra $(A,p,0,\muu,1)$. If $A$ contains $2$ such that $p(0,\muu,2)=1$ then $(X,+,e,\varphi)$ is a module over the unitary ring $(A,+,\cdot,0,1)$, with
\begin{eqnarray}
a+b&=&p(0,2,p(a,\muu,b))\\
a\cdot b&=&p(0,a,b)\\
\varphi_a(x)&=&q(e,a,x)\\
x+ y&=&q(e,2,q(x,\muu,y))=\varphi_2(q(x,\muu,y)).
\end{eqnarray}
\end{theorem}
\begin{proof}
$(A,+,\cdot,0,1)$ is a unitary ring by Theorem 7.1 of \cite{mobi}. We prove here that $(X,+,e,\varphi)$ is a module over $A$. First, we observe that, using in particular \ref{Y3} of Proposition \ref{properties_space}, we have:
\begin{eqnarray*}
q(e,\muu,x+y)&=& q(e,\muu,q(e,2,q(x,\muu,y)))\\
&=& q(e,\muu\cdot 2,q(x,\muu,y))\\
                            &=& q(e,1,q(x,\muu,y))\\
														&=& q(x,\muu,y).
\end{eqnarray*}
Then, the property (\ref{affine_1}) of an affine mobi space is essential to prove the associativity of the operation $+$ of the module:
\begin{eqnarray*}
q(e,\muu,q(e,\muu,(x+ y)+ z))&=&q(q(e,\muu,e),\muu,q(x+ y,\muu,z))\\
&=& q(q(e,\muu,x+ y),\muu,q(e,\muu,z))\\
&=& q(q(x,\muu,y),\muu,q(e,\muu,z))\\
&=& q(q(x,\muu,e),\muu,q(y,\muu,z))\\
&=& q(q(e,\muu,x),\muu,q(e,\muu,y+ z))\\
&=& q(q(e,\muu,e),\muu,q(x,\muu,y+ z))\\
&=& q(e,\muu,q(e,\muu,x+(y+ z)))\\
\end{eqnarray*}
which, by \ref{space_cancel}, implies that $(x+ y)+ z=x+(y+ z)$. Commutativity of $+$ and the identity nature of $e$ are easily proved:
\begin{eqnarray*}
q(e,\muu,e+ x)&=&q(e,\muu,x)\Rightarrow e+ x=x\\
q(e,\muu,x+ y)&=&q(x,\muu,y)=q(y,\muu,x)\\
                    &=&q(e,\muu,y+ x)\Rightarrow x+ y=y+ x.
\end{eqnarray*}
Cancellation is achieved with $-x=q(x,2,e)$. Indeed:
\begin{eqnarray*}
q(e,\muu,x+q(x,2,e))&=&q(x,\muu,q(x,2,e))\\
                    &=&q(q(x,0,e),\muu,q(x,2,e))\\
                    &=&q(x,p(0,\muu,2),e)\\
                    &=&q(x,1,e)\\
                    &=&e=q(e,\muu,e).\\
\end{eqnarray*}
To prove that $\varphi_a(x+y)=\varphi_a(x)+\varphi_a(y)$, we will again need (\ref{affine_1}):
\begin{eqnarray*}
q(e,\muu,\varphi_a(x+y))&=& q(e,\muu,q(e,a,x+ y))\\    
											 &=& q(q(e,a,e),\muu,q(e,a,x+ y))\\   
											 &=& q(q(e,\muu,e),a,q(e,\muu,x+ y))\\   
											 &=& q(e,a,q(x,\muu,y))\\ 
											 &=& q(q(e,a,x),\muu,q(e,a,y))\\  
											 &=& q(\varphi_a(x),\muu,\varphi_a(y))\\ 
											 &=& q(e,\muu,\varphi_a(x)+\varphi_a(y)).
\end{eqnarray*}
To prove that $\varphi_{a+b}(x)=\varphi_a(x)+\varphi_b(x)$, let us first recall that, in a mobi algebra with 2 and $a+b=p(0,2,p(a,\muu,b))$, we have the following property:
$$p(0,\muu,a+b)=p(a,\muu,b).$$
We then have
\begin{eqnarray*}
q(e,\muu,\varphi_{a+b}(x))&=& q(e,\muu,q(e,a+b,x))\\    
													&=& q(q(e,0,x),\muu,q(e,a+b,x))\\    
													&=& q(e,p(0,\muu,a+b),x)\\    
													&=& q(e,p(a,\muu,b),x)\\    
													&=& q(q(e,a,x),\muu,q(e,b,x))\\    
													&=& q(\varphi_a(x),\muu,\varphi_b(x))\\
													&=& q(e,\muu,\varphi_a(x)+\varphi_b(x)).
\end{eqnarray*}
The last two properties are easily proved:
$$\varphi_{a\cdot b}(x)=q(e,a\cdot b,x)=q(e,a,q(e,b,x))=\varphi_a(\varphi_b(x))$$
$$\varphi_1(x)=q(e,1,x)=x.$$
\end{proof}

\begin{proposition}\label{module2module}
Consider a R-module $(X,+,e,\varphi)$ within the conditions of Theorem \ref{module2mobi} and the corresponding mobi space $(X,q)$. Then the R-module obtained from $(X,q)$ by Theorem \ref{mobi2module} is the same as $(X,+,e,\varphi)$.
\end{proposition}
\begin{proof}
From $(X,q)$, we define
$$x+'y=q(e,2,q(x,\muu,y))\ \textrm{and}\ \varphi_a'(x)=q(e,a,x)$$
and obtain the following equalities:
\begin{eqnarray*}
x+'y&=&e+\varphi_2(q(x,\muu,y))\\
    &=&\varphi_2(\varphi_\muu(x)+\varphi_\muu(y))\\
		&=&\varphi_{2\cdot\muu}(x)+\varphi_{2\cdot\muu}(y)\\
		&=&x+y
\end{eqnarray*}
$$\varphi'_a(x)=\varphi_{1-a}(e)+\varphi_a(x)=e+\varphi_a(x)=\varphi_a(x).$$
\end{proof}

\begin{proposition}\label{mobi2mobi}
Consider an affine mobi space $(X,q)$ within the conditions of Theorem \ref{mobi2module} and the corresponding module $(X,+,e,\varphi)$. Then the affine mobi space obtained from $(X,+,e,\varphi)$ by Theorem \ref{module2mobi} is the same as $(X,q)$.
\end{proposition}
\begin{proof}
From $(X,+,e,\varphi)$, we define
$$ q'(x,a,y)=\varphi_{1-a}(x)+\varphi_a(y)$$
and obtain the following equalities:
\begin{eqnarray*}
q'(x,a,y)&=& q(e,\overline{a},x)+q(e,a,y)\\
         &=& q(x,a,e)+q(e,a,y)\\
         &=& q(e,2,q(q(x,a,e),\muu,q(e,a,y))).
\end{eqnarray*}
Now, because we are considering that $(X,q)$ is affine, we get:
\begin{eqnarray*}
q'(x,a,y)&=& q(q(e,a,e),2,q(q(x,\muu,e),a,q(e,\muu,y)))\\
         &=& q(q(e,2,q(e,\muu,x)),a,q(e,2,q(e,\muu,y)))\\
         &=& q(q(e,2\cdot \muu,x),a,q(e,2\cdot\muu,y))\\
				 &=& q(x,a,y).
\end{eqnarray*}
\end{proof}

We have completely characterized affine mobi spaces over a mobi algebra with $2$ 
in terms of modules over a unitary ring in which $2$ is invertible.
We end this section by taking a closer look to the examples of Section~\ref{sec_examples} trying to find the corresponding module if it exists. 
In Example~\ref{example1}, the mobi algebra does not contain the inverse of $\muu$ but, if we extend the mobi algebra to $(\R,p,0,\frac{1}{2},1)$, keeping the same $p$, then Theorem \ref{mobi2module} can be applied and the module obtained is the canonical vector space $\R^n$.
When considered over the real line, Examples \ref{projectiles} and~\ref{criticaldamping} may be obtained as a transport of the canonical structure (see Example~\ref{ex-general}), both having a physical interpretation \cite{preprint2}. More specifically, by Theorem \ref{mobi2module}, Example \ref{projectiles} leads to a module over $(\R,+,\cdot,0,1)$ given by $(\R^{n+1},+,0,\varphi)$ for $x,y,k\in\R^n$, $s,t\in\R$, with
\begin{eqnarray*}
(x,s)+(y,t)&=&(x+y-2 s t k,s+t)\\
\varphi_a(x,s)&=&(a x+a (1-a) s^2 k,a s).
\end{eqnarray*}
In the case of Example \ref{criticaldamping}, we get:
\begin{eqnarray*}
(x,s)+(y,t)&=&(x e^{\alpha t}+y e^{\alpha s},s+t)\\
\varphi_a(x,s)&=&(a x e^{-\alpha(1-a)s},a s).
\end{eqnarray*}
In both cases, using Theorem \ref{module2mobi}, we can construct a mobi space from the module and verify that it is the same as the original mobi space. 
Moreover, the following bijections 
\[f(x,s)=(x+ s (1-s) k,s)\]
and 
\[f(x,s)=(x\,e^{\alpha s},s)\]
can be used to map the canonical vector space $\R^{n+1}$ into the structures obtained in Example \ref{projectiles} and Example \ref{criticaldamping}, respectively.
Example \ref{ex7} does not contain the inverse of $\muu$ and is not well-defined outside the unit interval. Consequently, we are not able to get the structure of a module from it. The same happens to Example~\ref{ex_lozenge} but in this case it can be extended to Example~\ref{ex_lozenge2} where the inverse of $\muu$ is $(2,0)$. In this example, Theorem~\ref{mobi2module} gives, for \mbox{$e=0$}, the module $(X,+,e,\varphi)$ with the following operations for $a,b\in A$ and $x,y\in X$:
\begin{eqnarray*}
a\cdot b=(a_1 b_1+k\,a_2 b_2,a_1 b_2+a_2 b_1)& ; & a+b=(a_1+b_1,a_2+b_2)\\
                     \varphi_a(x)=(a_1+h\,a_2)\,x& ; & x+y=x+y.
\end{eqnarray*}
The set of scalars is a ring for all $k\in \R$ and a field for $k<0$.
In Example~\ref{ex-noncommut}, the inverse of $\muu$ is $(2,0,2)$. 
Theorem \ref{mobi2module} leads to a module over the ring $(\R^3,+,\cdot,(0,0,0),(1,0,1))$ given by $(\R^{2},+,(0,0),\varphi)$  with
\begin{eqnarray*}
a\cdot b=(a_1 b_1,a_1 b_2+a_2 b_3,a_3 b_3)&;&
a+b=(a_1+b_1,a_2+b_2,a_3+b_3)\\
\varphi_a(x)=(a_1 x_1+a_2 x_2,a_3 x_2)&;&
x+y=(x_1+y_1,x_2+y_2).
\end{eqnarray*}
Note that the ring here is isomorphic to the ring of $2\times2$ upper triangular matrices.

Theorem \ref{mobi2module} cannot be applied to Example \ref{ex-nonaffine} of Section \ref{sec_examples} because this example is not an affine mobi space. However, we can still find the corresponding $+$ operation and $\varphi_a$ functions. Choosing $e=(0,0)$, the results are the following:
\begin{eqnarray*}
\varphi_a(x,s)&=&\left(a\,x\,\frac{a s+i}{s+i},a s\right)\\
(x,s)+(y,t)&=&\left(x\,\frac{s+3t+2i}{s+t+2i}+y\,\frac{3s+t+2i}{s+t+2i},s+t\right).
\end{eqnarray*}
The operation $+$ is commutative, admits $e$ as an identity and the symmetric element $-(x,s)=\left(x\,\frac{s-i}{s+i},-s\right)$ for all $(x,s)\in X=\C^n\times\R$ but it is not associative and the following condition may not hold true:
$$\varphi_a\big((x,s)+(y,t)\big)=\varphi_{a}(x,s)+\varphi_{a}(y,t).$$
Example \ref{ex-nonaffine2} is also a non-affine mobi space and hence, if applying Theorem \ref{mobi2module}, we do not get a module over the real line.

Reanalysing the proofs of the previous Theorems we see that if condition~(\ref{affine_1}) is not satisfied then, for each choice of $e\in X$, we obtain a structure $(X,+,e,\varphi)$ that verifies all conditions for a module except associativity of $+$ and $\varphi_a$ being an homomorphism. However, this conditions may be satisfied for some particular cases of well chosen \emph{origins}~$e\in X$. This leads naturally to the study of which elements $e\in X$ can be chosen as the origin of a module. In \cite{ccm_magmas} this approach was used for the study of internal monoids in the category of midpoint algebras.

\section{The isomorphism between pointed affine mobi spaces and modules over a ring}\label{isomorphism}

We will now use the previous results to show that there is an isomorphism of categories between modules over a ring with $\muu$ and pointed affine mobi spaces over a mobi algebra with $2$.

The category of modules over a ring (with $\muu$) is the category whose objects are pairs $(R,X_R)$ in which $R$ is a unitary ring with $\muu$, the inverse of $2$, and 
$X_R=(X,+,e,\varphi)$ is a module over $R$, as in the previous section. Morphisms are pairs $(f,g)$ such that $f\colon{R\to R'}$ is a ring homomorphism and $g\colon{X\to X'}$ is such that
\begin{eqnarray}
g(e)&=&e\\
g(x+y)&=&g(x)+g(y)\\
g(\varphi_a(x))&=&\varphi_{f(a)}(g(x)).
\end{eqnarray} 

The category of pointed affine mobi spaces consists of pairs $(A,X_A)$ in which $A$ is a mobi algebra with $2$, such that $p(0,\muu,2)=1$, and $X_A=(X,q,e)$ is an affine mobi space $(X,q)$ over $A$ together with an element $e\in X$ considered as its origin. The morphisms in this category are pairs $(f,g)$ such that $f\colon{A\to A'}$ is a morphism of mobi algebras and $g\colon{X\to X'}$ is such that
\begin{eqnarray}
g(e)&=&e\\
g(q(x,a,y))&=&q'(g(x),f(a),g(y)).
\end{eqnarray} 

Note that the category of pointed affine mobi spaces can be seen as a slice category $(*\downarrow \textbf{Aff})$ and it is an example of an additive category with kernels \cite{Carboni,CarboniJanelidze}.
We are now in position to prove our main result.

\begin{theorem}\label{thm:main}
There is an isomorphism of categories between modules over a ring with $\muu$ and pointed affine mobi spaces over a mobi algebra with $2$.
\end{theorem}
\begin{proof}
Theorem \ref{module2mobi} gives us a functor from modules to pointed affine mobi spaces defined on objects. It remains to show that morphisms of modules are transformed into morphisms of pointed affine mobi spaces. Indeed, let $(f,g)$ be a homomorphism of modules, then:
\begin{eqnarray*}
f\left(p(a,b,c)\right)&=&f(a-b\cdot a+b\cdot c)\\
                      &=&f(a)-f(b)\cdot f(a)+f(b)\cdot f(c)\\
											&=&p'\left(f(a),f(b),f(c)\right)
\end{eqnarray*}
\begin{eqnarray*}
g\left(q(x,a,y)\right)&=&g\left(\varphi_  {1-a}(x)+\varphi_a(y)\right)\\
                      &=&g\left(\varphi_  {1-a}(x)\right)+g\left(\varphi_a(y)\right)\\
											&=&\varphi_{f(1-a)}(g(x))+\varphi_{f(a)}(g(y))\\
											&=&\varphi_{1-f(a)}(g(x))+\varphi_{f(a)}(g(y))\\
											&=&q'\left(g(x),f(a),g(y)\right).
\end{eqnarray*}

Theorem \ref{mobi2module} gives us a functor from pointed affine mobi spaces into modules over a ring defined on objects. It remains to show that morphisms of  pointed affine mobi spaces are transformed into morphisms of modules over a ring. Let $(f,g)$ be a morphism of affine mobi spaces with $2$. First, we observe that $f(2)=2$ because, from the definition of $2$, we have $f\left(p(0,\muu,2)\right)=f(1)$ and consequently $p'(0,\muu,f(2))=1=p'(0,\muu,2)$ which gives the result by application of Axiom \ref{alg_cancel}. We then have:
\begin{eqnarray*}
f(a+b)&=&f\left(p(0,2,p(a,\muu,b))\right)\\
      &=&p'(0,2,p'(f(a),\muu,f(b)))\\
			&=&f(a)+f(b)
\end{eqnarray*}
\begin{eqnarray*}
f(a\cdot b)&=&f\left(p(0,a,b)\right)\\
      &=&p'(0,f(a),f(b))\\
			&=&f(a)\cdot f(b)
\end{eqnarray*}
\begin{eqnarray*}
g\left(\varphi_a(x)\right)&=&g\left(q(e,a,x)\right)\\
      &=&q'(e,f(a),g(x)))\\
			&=&\varphi'_{f(a)}\left(g(x)\right)
\end{eqnarray*}
\begin{eqnarray*}
g\left(x+y\right)&=&g\left(q(e,2,q(x,\muu,y))\right)\\
      &=&q'(e,2,q'(g(x),\muu,g(y)))\\
			&=&g(x)+g(y).
\end{eqnarray*}
Theorems \ref{module2module} and \ref{mobi2mobi} show that the two constructions are inverse to each other thus defining an isomorphism of categories.
\end{proof}


\section{Conclusion}
There are several different ways to look at the notion of an affine space (see e.g. \cite{Bertram, bourn_geometrie,Porteus}). Roughly speaking it may be seen as a set of points in which every point can be chosen as the origin of a vector space. We have shown that in the case where the scalars (elements of the mobi algebra) form a field with characteristic different from $2$, our concept of affine mobi space is precisely an affine space in the usual sense. Furthermore, it is possible to consider a ring of scalars rather than a field and our results show that multiplication does not need to be commutative. In any case, the existence of $2$ and $\muu$ is required in order to define the module addition $x+y=q(e,2,q(x,\muu,y))$ and prove that an affine mobi space $(X,q)$ admits a module structure 
for every choice of an origin $e\in X$, with scalar multiplication $a x=q(e,a,x)$.

In \cite{inprogress1} it is shown that the isomorphism between unitary rings with one-half and mobi algebras with one double is part of a more general connection between the categories of unitary rings and mobi algebras, thus suggesting that the previous result can be extended too. Another line of investigation which can be developed further is to analyse under which conditions the structure of an affine mobi space $(X,q)$, over a mobi algebra $A$, is determined by a mobi algebra homomorphism $f\colon{A\to \End(X)}$ for an appropriate mobi algebra structure on the set of endomorphisms of $X$ (see Proposition 2.1 in \cite{preprint1}).



\begin{thebibliography}{99}

\bibitem{Bertram} W. Bertram, \emph{From linear algebra via affine algebra to projective algebra}, Linear Algebra and its Applications \textbf{378} (2004), 
109--134.

\bibitem{bourn_geometrie} D. Bourn, \emph{Traité de géométrie affine}, Ellipses, 168 pages, 2012.

\bibitem{Carboni}  A. Carboni, \emph{Categories of Affine Spaces}, Journal of Pure and Applied Algebra 61 (1989), 243--250.

\bibitem{CarboniJanelidze} A. Carboni and G. Janelidze, \emph{Modularity and Descent}, Journal of Pure and Applied Algebra 99 (1995) pp.255--265.





\bibitem{ccm_magmas} J. P. Fatelo and N. Martins-Ferreira, \emph{Internal monoids and groups in the category of commutative cancellative medial magmas}, Portugaliae Mathematica, Vol. \textbf{73}, Fasc. 3 (2016) 219--245. https://doi.org/10.4171/PM/1986.



\bibitem{mobi} J. P. Fatelo and N. Martins-Ferreira, \emph{Mobi algebra as an abstraction to the unit interval and its comparison to rings}, Communications in Algebra \textbf{47} (3) (2019) 1197--1214. https://doi.org/10.1080/00927872.2018.1501575.

\bibitem{preprint1} J. P. Fatelo and N. Martins-Ferreira, \emph{Mobility spaces and their geodesic paths}, arXiv:2001.03441v1,2020.

\bibitem{preprint2} J. P. Fatelo and N. Martins-Ferreira, \emph{Mobi spaces and geodesics for the n-sphere}, arXiv:2102.02692, 2021.

\bibitem{inprogress1} J. P. Fatelo and N. Martins-Ferreira, \emph{A connection between unitary rings and mobi algebras}, Scripta-Ingenia \textbf{10} (2021) 3--5.



















\bibitem{Porteus} I. R. Porteus, \emph{topological Geometry}, 2ed Edition, Cambridge University Press, New York, 1981.

\end{thebibliography}
\end{document}